\documentclass[12pt, a4paper]{amsart}

\usepackage{ifthen}
\usepackage{amsfonts}
\usepackage{amssymb}
\usepackage{array}
\usepackage{eucal}
\usepackage[margin=1.2in]{geometry}
\usepackage{xcolor}
\definecolor{cite}{rgb}{0.50,0.00,1.00}
\definecolor{url}{rgb}{0.00,0.50,0.75}
\definecolor{link}{rgb}{0.00,0.00,0.50}
\usepackage[colorlinks,linkcolor=link,urlcolor=url,citecolor=cite,pagebackref,breaklinks]{hyperref}
\usepackage{indentfirst}
\usepackage{lscape}
\usepackage{mathtools}
\usepackage{mathrsfs}
\usepackage{bbm}
\usepackage[all]{xypic}
\usepackage[lite,abbrev,msc-links]{amsrefs}

\newcommand{\GL}{{\mathrm{GL}}}

\newcommand{\rk}{{\mathrm{k}}}

\newcommand{\od}{\operatorname{d}}

\renewcommand{\rk}{\mathrm k}

\newcommand{\C}{\mathbb{C}}

\newcommand{\abs}[1]{\lvert#1\rvert}

\newcommand{\la}{\langle}
\newcommand{\ra}{\rangle}
\newcommand{\be}{\begin {equation}}
\newcommand{\ee}{\end {equation}}
\newcommand{\bp}{\begin {proof}}
\newcommand{\ep}{\end {proof}}
\newcommand{\bee}{\begin {equation*}}
\newcommand{\eee}{\end {equation*}}

\theoremstyle{Theorem}

\newtheorem{thm}{Theorem}[section]

\newtheorem{lemt}[thm]{Lemma}
\newtheorem{prpt}[thm]{Proposition}
\newtheorem{thmt}[thm]{Theorem}
\newtheorem{remt}[thm]{Remark}

\theoremstyle{Theorem}

\theoremstyle{Theorem}

\theoremstyle{Plain}

\theoremstyle{Definition}

\begin{document}

\title{Vector-valued Gelfand-Kazhdan criterion}

\author{Fulin Chen}\address{School of Mathematical Sciences, Xiamen University,
	Xiamen, 361005, China} \email{chenf@xmu.edu.cn}

\author{Binyong Sun}
\address{Institute for Advanced Study in Mathematics and New Cornerstone Science Laboratory, Zhejiang University,  Hangzhou, 310058, China}
\email{sunbinyong@zju.edu.cn}

\author{Yixiang Weng}
\address{School of Mathematical Sciences, Zhejiang University,  Hangzhou, 310058, China}
\email{12235036@zju.edu.cn}


\begin{abstract}
The Gelfand–Kazhdan criterion is a fundamental tool for studying multiplicity-one properties of local periods of representations. However, it does not apply to many cases arising in the relative Langlands program. Generalizing the usual Gelfand–Kazhdan criterion, we formulate and prove a vector-valued Gelfand–Kazhdan criterion that fits into the general framework of the relative Langlands program. As an illustration of its effectiveness, we establish the multiplicity-one property for the local Asai Rankin–Selberg periods.
\end{abstract}

\maketitle

\section{Introduction and main results}
\label{sect:intro}

One of the major problems in representation theory is the branching problem, which studies the occurrence of representations of a smaller group within representations of a larger group.  

Let $G$ be a topological group. By a representation of $G$, we mean a quasi-complete, locally convex, Hausdorff topological vector space over $\C$ together with a continuous linear $G$-action on it. For abbreviation and as usual, we do not distinguish a representation from its underlying vector space. Let $H$ be another topological group, endowed with a fixed representation $\omega$ of it. Suppose that we are given an injective continuous homomorphism $H\rightarrow G$, through which $H$ is regarded as a subgroup of $G$. For every representation $\pi$ of $G$, the general branching problem is to understand the space
\[
\mathrm B_H(\pi\times \omega,\C):=\{\textrm{$H$-invariant separately continuous bilinear forms on $\pi\times \omega$}\}.
\]
We call an element of the above space an $\omega$-period of $\pi$. 

From now on, we assume that $G$ is either a real reductive group or a second-countable totally disconnected locally compact Hausdorff topological group. We say that $(G, \omega)$ is a Gelfand pair, or that it has the multiplicity-one property, if
\[
\dim \mathrm B_H(\pi\times \omega,\C) \leq 1 
\]
for all irreducible admissible smooth representations $\pi$ of $G$. Here, when $G$ is a real reductive group, ``admissible smooth representations" means Casselman–Wallach representations (see \cite{Ca}, \cite[Chapter 11]{Wa}, or \cite{BK}); when $G$ is totally disconnected, the term retains its usual meaning. In the latter case, the representation $\pi$ is equipped with the finest locally convex topology.

When $\omega$ is one-dimensional, the multiplicity-one property for a wide range of pairs $(G, \omega)$ has been extensively studied; see, for instance, \cite{AG, AGS, AGS2, AGRS, BD, D, F91, JR, JSZ, LST, Sh, Sun, SZ, Z}. A unifying theme in these works is the use of the Gelfand–Kazhdan criterion (see \cite{GK, SZ2}), which reduces the verification of the multiplicity-one property to the vanishing of certain equivariant generalized functions.

In the general formulation of the relative Langlands program introduced by Ben-Zvi, Sakellaridis, and Venkatesh (see \cite{BZSV}), particular interest lies in cases where $(H, \omega)$ arises from a hyperspherical Hamiltonian space (see \cite[Section 3]{BZSV} and \cite[Definition 1.2]{MWZ}). In this setting, $H$ admits a semidirect product decomposition $H_0 \ltimes U$, where $H_0$ is a reductive group and $U$ is a unipotent group associated with a certain $\mathfrak{sl}_2$-triple. Moreover, $\omega$ is realized on an irreducible smooth representation of a certain Heisenberg group attached to this $\mathfrak{sl}_2$-triple and a certain symplectic representation of $H_0$. In particular, $\omega$ is often infinite-dimensional.

As emphasized in \cite[Remark 1.2.2]{BZSV}, a central feature of the framework of Ben-Zvi, Sakellaridis, and Venkatesh is the notion of a hyperspherical variety, which is obtained by abstracting the role of multiplicity-one in the Langlands program. They further note that when $(G, \omega)$ is a Gelfand pair, $\omega$-periods are closely related to $L$-functions. Consequently, from the perspective of studying automorphic $L$-functions, it is important to determine which pairs $(G, \omega)$ arising from hyperspherical Hamiltonian spaces are Gelfand pairs.

As mentioned earlier, the principal tool in the literature for establishing multiplicity one is the Gelfand–Kazhdan criterion. However, this criterion applies only when $\omega$ is one-dimensional and therefore does not cover the more general situations that arise in the framework of \cite{BZSV}. In this note, we formulate and prove a \emph{vector-valued Gelfand–Kazhdan criterion} (see Theorem \ref{thm:main2}) that applies to representations $\omega$ of arbitrary dimension, thereby providing a genuine generalization of the classical criterion.  

To state the main result of this note, we introduce further notation. Denote by $\mathrm{C}^{\varsigma}(G)$ the space of Schwartz functions on $G$. More precisely, if $G$ is totally disconnected, then
\[
\mathrm{C}^{\varsigma}(G):= \{\textrm{locally constant $\C$-valued functions on $G$ with compact support}\}; 
\]
while if $G$ is a real reductive group, then $\mathrm{C}^{\varsigma}(G)$ consists of all $\C$-valued smooth functions $f$ on $G$ such that 
\begin{equation}\label{Nash}
\sup_{x\in G}\, \abs{\varphi(x)\cdot (Df)(x)}<\infty
\end{equation}
for every left-invariant differential operator $D$ on $G$ and every $\C$-valued smooth function $\varphi$ on $G$ whose left $G$-translations span a finite-dimensional space.

Define the space of Schwartz densities on $G$ by 
\[
\mathrm D^{\varsigma}(G):= \mathrm{C}^{\varsigma}(G)\cdot (\textrm{a fixed left-invariant Haar measure on $G$}).
\]
This is a space of measures on $G$ and forms a $\C$-algebra (without identity, in general) under convolution. When $G$ is a real reductive group, $\mathrm D^{\varsigma}(G)$ is naturally a Fréchet space; when $G$ is totally disconnected, $\mathrm D^{\varsigma}(G)$ is equipped with the finest locally convex topology, as before.

Let $H_1$ and $H_2$ be two topological groups, with fixed injective continuous homomorphisms $H_1\rightarrow G$ and $H_2\rightarrow G$. Let $\omega_1$ and $\omega_2$ be representations of $H_1$ and $H_2$, respectively. Then the space $\mathrm D^\varsigma(G)\times \omega_1\times \omega_2$ carries a natural left action of the group $H_1\times H_2$, where $H_1$ acts on $\mathrm D^{\varsigma}(G)$ by left translations and $H_2$ acts on $\mathrm D^{\varsigma}(G)$ by right translations.

\begin{thm}\label{thmgk}
Suppose that there exists a topological anti-automorphism $\sigma: G\rightarrow G$ and topological linear isomorphisms $\sigma_1: \omega_1\rightarrow \omega_2$ and $\sigma_2: \omega_2\rightarrow \omega_1$ with the following property: every $H_1\times H_2$-invariant separately continuous trilinear form
\[
\mathrm D^\varsigma(G)\times \omega_1\times \omega_2 \rightarrow \C
\]
is invariant under the linear automorphism
\begin{equation}\label{switch}
\mathrm D^\varsigma(G)\times \omega_1\times \omega_2\rightarrow \mathrm D^\varsigma(G)\times \omega_1\times \omega_2, \quad (\eta, u_1, u_2)\mapsto (\sigma(\eta), \sigma_2(u_2), \sigma_1(u_1)).
\end{equation}
Then, for every irreducible admissible smooth representation $\pi$ of $G$, 
\[
\dim \mathrm{B}_{H_1}(\pi\times \omega_1, \C) \cdot \dim \mathrm{B}_{H_2}(\pi^\vee \times \omega_2, \C)\leq 1.
\]
\end{thm}

Here $\pi^\vee$ denotes the contragredient of $\pi$, and $\sigma:\mathrm D^\varsigma(G)\rightarrow \mathrm D^\varsigma(G)$ denotes the push-forward of Schwartz densities by $\sigma: G \to G$.

We say that a topological automorphism $\sigma': G \to G$ is a \emph{Chevalley automorphism} if for every irreducible admissible smooth representation $\pi$ of $G$ there exists a topological linear isomorphism $\pi \to \pi^\vee$ that is equivariant with respect to $\sigma'$.

The following lemma is immediate.

\begin{lemt}\label{thmgk2}
Suppose that there exist a Chevalley automorphism $\sigma': G \to G$ and a topological linear isomorphism $\varphi: \omega_1 \to \omega_2$ such that 
\begin{itemize}
\item $\sigma'(H_1) = H_2$; and
\item $\varphi$ is equivariant with respect to the induced isomorphism $\sigma': H_1 \to H_2$.
\end{itemize}
Then, for every irreducible admissible smooth representation $\pi$ of $G$,
\[
\dim \mathrm{B}_{H_1}(\pi \times \omega_1, \C) = \dim \mathrm{B}_{H_2}(\pi^\vee \times \omega_2, \C).
\]
\end{lemt}

Theorem \ref{thmgk} and Lemma \ref{thmgk2} yield the following direct consequence, which we refer to as the \emph{vector-valued Gelfand–Kazhdan criterion}.

\begin{thmt}\label{thm:main2}
Under the assumptions of Theorem \ref{thmgk} and Lemma \ref{thmgk2}, both $(G, \omega_1)$ and $(G, \omega_2)$ are Gelfand pairs. 
\end{thmt}

To illustrate the utility of the vector-valued Gelfand–Kazhdan criterion, we apply it to prove the following multiplicity-one property for the local Asai Rankin–Selberg periods (see \cite{F, K, BP}).

\begin{thmt}\label{thm:main3} 
Let $n$ be a positive integer, and let $\rk'\,/\,\rk$ be a quadratic extension of local fields of characteristic zero. Then $(\GL_n(\rk'), \mathrm{C}^{\varsigma}(\rk^{1\times n}))$ is a Gelfand pair, where $\mathrm{C}^\varsigma(\rk^{1 \times n})$ is the space of complex-valued Schwartz functions on $\rk^{1 \times n}$ (the space of row vectors), to be viewed as a representation of $\GL_n(\rk)$ by
\begin{eqnarray}\label{introeq1}
(h.\phi)(x)=\phi(x h),\quad h\in \GL_n(\rk), \phi\in\mathrm{C}^\varsigma(\rk^{1 \times n}), x\in \rk^{1\times n}.
\end{eqnarray}
\end{thmt}

\begin{remt}
Theorem \ref{thm:main3} remains valid when $\rk'$ is replaced by the split extension $\rk \times \rk$. In the split case, Theorem \ref{thm:main3} is known as the uniqueness of Fourier–Jacobi periods and was proved in \cite{Sun} for the non-archimedean case and in \cite{SZ} for the archimedean case. However, the Gelfand–Kazhdan criterion used in \cite{Sun, SZ} does not apply to Theorem \ref{thm:main3}.
\end{remt}

For possible future applications, we also prove the following result concerning the disjointness of local periods, which generalizes \cite[Theorem 2.3(b)]{SZ2}.

\begin{prpt}\label{prpgk}
Suppose that every $H_1\times H_2$-invariant separately continuous trilinear form
\[
\mathrm D^\varsigma(G)\times \omega_1\times \omega_2 \rightarrow \C
\]
is zero. Then for every irreducible admissible smooth representation $\pi$ of $G$,
\[
\dim \mathrm{B}_{H_1}(\pi\times \omega_1, \C) \cdot \dim \mathrm{B}_{H_2}(\pi^\vee \times \omega_2, \C)=0.
\]
\end{prpt}

\begin{remt}In \cite{Sun2}, the second named author formulates and proves a  version of Gelfand-Kazhdan criterion  in the setting of real Jacobi groups.
By the argument of this note,  Theorem \ref{thmgk}, Lemma \ref{thmgk2}, Theorem  \ref{thm:main2}, and Proposition   \ref{prpgk} remain valid when $G$ is replaced by  a real Jacobi group  and the Casselman-Wallach representations are replaced by the Casselman-Wallach
$\psi$-representations  (in the sense of  \cite{Sun2}). 
\end{remt}
\section{Proof of Theorem \ref{thmgk} and Proposition \ref{prpgk}}

This section is devoted to proving Theorem \ref{thmgk} and Proposition \ref{prpgk}. 
We begin with the following elementary result from linear algebra, whose proof is omitted.

\begin{lemt}\label{leme1}
Let $E_1$ and $E_2$ be vector spaces over a field $\mathrm{k}_0$, and let $\lambda_1, \lambda_2: E_1\times E_2 \rightarrow \mathrm{k}_0$ be bilinear forms. Suppose that for all $v_1\in E_1$ and $v_2\in E_2$,
\[
  \lambda_1(v_1, v_2)=0 \quad \text{implies}\quad \lambda_2(v_1, v_2)=0.
\]
Then there exists a scalar $a\in \mathrm{k}_0$ such that $\lambda_2 = a \lambda_1$.
\end{lemt}

Every admissible smooth representation $\pi$ of $G$ carries an (abstract) $\mathrm D^\varsigma(G)$-module structure given by
\[
 \eta.u:=\int_G f(g) g.u\,\od \!g, \quad \eta=f \od\!g, \ u\in \pi,
\]
where $f\in \mathrm{C}^{\varsigma}(G)$ and $\od\!g$ is a left-invariant Haar measure on $G$. Let $\pi'$ denote the space of continuous linear functionals on $\pi$, equipped with the strong topology (i.e., the topology of uniform convergence on bounded sets). It is again a $\mathrm D^\varsigma(G)$-module, with the module structure defined by 
\[
\langle \eta.\lambda ,u\rangle =\langle \lambda, \eta^\vee.u\rangle,\quad \eta\in \mathrm D^\varsigma(G), \,\lambda\in \pi',\ u\in \pi,
\]
where $\eta^\vee$ denotes the pushforward of $\eta$ through the inversion map $G\rightarrow G,\ x\mapsto x^{-1}$. 

Set $\pi^{-\infty}:=(\pi^\vee)'$. Then $\pi$ is identified with a $\mathrm D^\varsigma(G)$-submodule of $\pi^{-\infty}$ via the obvious injective linear map $\pi\rightarrow \pi^{-\infty}$. It is known that (cf. \cite[Lemma 3.5]{SZ2})
\begin{equation}\label{dpi}
\mathrm D^\varsigma(G).\pi^{-\infty}=\mathrm D^\varsigma(G).\pi=\pi,
\end{equation}
and that the map
\begin{equation}\label{sepc}
  \mathrm D^\varsigma(G)\times \pi^{-\infty}\rightarrow \pi \quad\text{is separately continuous.}
\end{equation}

\begin{lemt}\label{lemsch00}
Let $\pi$ be an irreducible admissible smooth representation of $G$. Then for every nonzero $\lambda\in \pi^{-\infty}$, 
\[
 \mathrm D^\varsigma(G).\lambda=\pi.
\]
\end{lemt}

\begin{proof}
The equalities in \eqref{dpi} imply that $\mathrm D^\varsigma(G).\lambda$ is a nonzero $\mathrm D^\varsigma(G)$-submodule of $\pi$. Note that the representation $\pi$ is irreducible as an (abstract) $\mathrm D^\varsigma(G)$-module (see \cite[Corollary 3.1.5]{Du} in the real reductive group case and \cite{BZ} in the totally disconnected case). The lemma then follows.  
\end{proof}

\begin{lemt}\label{lemsch}
Let $\pi$ be an irreducible admissible smooth representation of $G$, and let $\phi, \phi'$ be two continuous linear functionals on $\pi^\vee$. Suppose that for all $\eta\in \mathrm D^{\varsigma}(G)$,
\[
 \eta.\phi=0 \quad \text{implies}\quad \eta.\phi'=0.
\]
Then $\phi'$ is a scalar multiple of $\phi$.
\end{lemt}

\begin{proof}
Without loss of generality, assume that $\phi\neq 0$ and $\phi'\neq 0$. 
Then, by \eqref{dpi} together with the assumption of the lemma, there is a well-defined $G$-equivariant linear endomorphism 
\[
 \pi\rightarrow \pi, \quad \eta.\phi\mapsto \eta.\phi'.
\]
By using \eqref{sepc}, Lemma \ref{lemsch00}, and the open mapping theorem, we know that this endomorphism is continuous. Schur's lemma (see \cite[Lemme 3.2.9]{Du} and \cite[Proposition 2.11]{BZ}) therefore implies that it is a scalar multiplication. That is, there exists $a\in \mathbb{C}$ such that 
\[
\eta.\phi'=a\cdot (\eta.\phi)\quad \text{for all $\eta\in \mathrm D^{\varsigma}(G)$}.
\]
Applying Lemma \ref{lemsch00} once more, we conclude that $\phi'=a\, \phi$.
\end{proof}

We are now ready to prove Theorem \ref{thmgk}. Retain the notation of Theorem \ref{thmgk}. 
Without loss of generality, we assume that both spaces  
\[
\mathrm{B}_{H_1}( \pi \times  \omega_1, \mathbb{C})\quad\text{and}\quad \mathrm{B}_{H_2}( \pi^\vee \times  \omega_2, \mathbb{C})
\] 
are nonzero. Choose nonzero elements $\lambda_1$ and $\lambda_2$ in these two spaces, respectively.

We define a trilinear form
\begin{equation}\label{defc}
c_{\lambda_1\otimes \lambda_2}:\  \mathrm D^\varsigma(G)\times \omega_1\times \omega_2 \rightarrow \mathbb{C},\quad (\eta, v_1, v_2)\mapsto \langle \lambda_1|_{v_1}, \eta.(\lambda_2|_{v_2})\rangle,
\end{equation}
where $\lambda_1|_{v_1}$ denotes the continuous linear functional 
\[
\pi\rightarrow \mathbb{C}, \quad u\mapsto \lambda_1(u, v_1),
\]
and similar notation will be used without further explanation. 

Note that the trilinear form $c_{\lambda_1\otimes \lambda_2}$ is separately continuous (cf. \cite[Lemma 3.5]{SZ2}) and $H_1\times H_2$-invariant. From the assumption of Theorem \ref{thmgk}, it follows that 
\begin{equation}\label{cinv}
c_{\lambda_1\otimes \lambda_2}\ \text{is invariant under the automorphism \eqref{switch} of}\ \mathrm D^\varsigma(G)\times \omega_1\times \omega_2.
\end{equation} 

Fix $v_2\in \omega_2$. Then for all $\eta\in \mathrm D^\varsigma(G)$,
\begin{eqnarray*}
  &&\eta. (\lambda_2|_{v_2}) =0 \\
  &\Leftrightarrow  & \la \lambda_1|_{v_1}, (\eta'*\eta).(\lambda_2|_{v_2})\ra = 0 \ \textrm{for all $v_1\in \omega_1$ and $\eta'\in \mathrm D^\varsigma(G)$}\quad (\text{by Lemma \ref{lemsch00}})\\
  &\Leftrightarrow & \la \lambda_1|_{\sigma_2(v_2)}, (\sigma(\eta'*\eta)).(\lambda_2|_{\sigma_1(v_1)})\ra = 0 \  \textrm{for all $v_1\in \omega_1$ and $\eta'\in \mathrm D^\varsigma(G)$}\quad (\textrm{by \eqref{cinv}})\\
  &\Leftrightarrow  & (\sigma(\eta))^\vee. (\lambda_1|_{\sigma_2(v_2)})=0\quad (\text{by Lemma \ref{lemsch00}}).
\end{eqnarray*}
Now let $\lambda'_2$ be another nonzero element of $\mathrm{B}_{H_2}( \pi^\vee \times  \omega_2, \mathbb{C})$. By the same argument, for all $\eta\in \mathrm D^\varsigma(G)$,
\[
\eta. (\lambda'_2|_{v_2})=0 \
\Leftrightarrow\ (\sigma(\eta))^\vee. (\lambda_1|_{\sigma_2(v_2)})=0.
\]
Consequently,
\[
\eta. (\lambda'_2|_{v_2}) =0 \
\Leftrightarrow\ \eta. (\lambda_2|_{v_2}) =0.
\]
By Lemma \ref{lemsch}, this implies that $\lambda_2|_{v_2}$ and $\lambda'_2|_{v_2}$ have the same kernel (as linear functionals on $\pi^\vee$). Since $v_2\in \omega_2$ is arbitrary, Lemma \ref{leme1} implies that $\lambda_2'$ is a scalar multiple of $\lambda_2$. Hence the space $\mathrm{B}_{H_1}(\pi\times \omega_1, \mathbb{C})$ is one-dimensional. 

A completely analogous argument shows that the space $\mathrm{B}_{H_2}(\pi^\vee\times \omega_2, \mathbb{C})$ is also one-dimensional. This finishes the proof of Theorem \ref{thmgk}.

Finally, we give a proof of Proposition \ref{prpgk}. Assume by contradiction that there exist nonzero elements 
\[
\lambda_1\in \mathrm{B}_{H_1}( \pi \times  \omega_1, \mathbb{C})\quad \text{and}\quad \lambda_2\in \mathrm{B}_{H_2}( \pi^\vee \times  \omega_2, \mathbb{C}).
\] 
Then the $H_1\times H_2$-invariant separately continuous trilinear form $c_{\lambda_1\otimes \lambda_2}$ (see \eqref{defc}) is nonzero by Lemma \ref{lemsch00}, which yields the desired contradiction.

\section{Proof of Theorem \ref{thm:main3}}
This section is devoted to proving Theorem \ref{thm:main3}. 
Let $N$ be either a Nash group or a totally disconnected, locally compact Hausdorff topological group.
We denote by $\mathrm{C}^\varsigma(N)$ the space of Schwartz functions on $N$. When $N$ is a reductive Nash group, this agrees with the space defined in \eqref{Nash}. We refer the interested reader to \cite{Shi,Sun3,AG0} for generalities on Nash manifolds, Nash groups, and their function spaces.

As before, the space of Schwartz densities on $N$ is defined by 
\[
\mathrm{D}^\varsigma(N)=\mathrm{C}^\varsigma(N)\cdot (\text{a fixed left-invariant Haar measure on $N$}). 
\]
When $N$ is a Nash group, both $\mathrm{C}^\varsigma(N)$ and $\mathrm{D}^\varsigma(N)$ are naturally Fr\'echet spaces.
By a \emph{tempered generalized function} on $N$, we mean a continuous linear functional on $\mathrm{D}^\varsigma(N)$. 

Let $n$ be a positive integer, and let $\rk'/\rk$ be a quadratic extension of local fields of characteristic zero as in Theorem \ref{thm:main3}. 
Recall that $\rk^{1\times n}$ denotes the space of row vectors, and similarly we write $\rk^{n\times 1}$ for the space of column vectors. 
By applying the vector-valued Gelfand--Kazhdan criterion, we will show at the end of this section that Theorem \ref{thm:main3} follows from the result stated below. 

\begin{prpt}\label{prop:group}
Let $f$ be a tempered generalized function on $\mathrm{GL}_n(\rk') \times \rk^{n\times 1} \times \rk^{1\times n}$ such that for all $h_1,h_2\in \mathrm{GL}_n(\rk)$,
\[
f(h_1gh_2^{-1}, h_1u, vh_2^{-1}) = f(g, u ,v), \quad g\in \mathrm{GL}_n(\rk'), \; u\in \rk^{n\times 1}, \; v\in \rk^{1\times n}.
\]
Then 
\[
f(g, u, v) = f(g^t, v^t, u^t).
\] 
\end{prpt}

Here and throughout, a superscript ``$t$" indicates the transpose of a matrix, and the equalities are understood as equalities of tempered generalized functions. 
Using linearization, we will reduce Proposition \ref{prop:group} to the following result. 

\begin{prpt}\label{prop:algebra}
Let $f$ be a tempered generalized function on $\mathfrak{gl}_n(\rk) \times \rk^{n\times 1} \times \rk^{1\times n}$ such that for all $g\in \mathrm{GL}_n(\rk)$,
\[
f(gXg^{-1}, gu, vg^{-1}) = f(X, u ,v), \quad X\in \mathfrak{gl}_n(\rk), \; u\in \rk^{n\times 1}, \; v\in \rk^{1\times n}.
\]
Then 
\[
f(X, u, v) = f(X^t, v^t, u^t).
\]  
\end{prpt}

\begin{proof}
The proposition is proved in \cite{AG,SZ} for the archimedean case and in \cite{AGRS} for the non-archimedean case.
\end{proof}

We now present a reformulation of Proposition \ref{prop:algebra} (cf.~\cite{AG,AGRS,SZ}). Write 
\[
\breve{\mathrm{GL}}_n(\rk)=\mathrm{GL}_n(\rk)\rtimes \{ \pm1\},
\]
where the semi-direct product is defined by the action 
\[
-1.h := h^{-t}, \quad h\in \mathrm{GL}_n(\rk).
\]
Let $\breve{\mathrm{GL}}_n(\rk)$ act on $\mathrm{GL}_n(\rk)$ by 
\[
(h,\delta).g := \begin{cases} 
h g h^{-1}, &\text{if } \delta=1; \\
h g^t h^{-1}, &\text{if } \delta=-1,
\end{cases}
\]
and on $\mathfrak{gl}_n(\rk)$ through its differential by 
\begin{equation}\label{eq:Hactgl}
(h,\delta).X := \begin{cases} 
h X h^{-1}, &\text{if } \delta=1;\\
h X^t h^{-1}, &\text{if } \delta=-1,
\end{cases}
\end{equation}
where $(h,\delta)\in \breve{\mathrm{GL}}_n(\rk)$, $g\in \mathrm{GL}_n(\rk)$, and $X\in \mathfrak{gl}_n(\rk)$.
Let $\breve{\mathrm{GL}}_n(\rk)$ act on $\rk^{n\times 1}\times \rk^{1\times n}$ by 
\[
(h,\delta).(u,v) := \begin{cases}
(hu, vh^{-1}), &\text{if } \delta=1;\\
(hv^t, u^th^{-1}), &\text{if } \delta=-1,
\end{cases}
\]
where $(u,v)\in \rk^{n\times 1}\times \rk^{1\times n}$. Finally, let $\breve{\mathrm{GL}}_n(\rk)$ act diagonally on $\mathfrak{gl}_n(\rk)\times \rk^{n\times 1}\times \rk^{1\times n}$.

Denote by 
\[
\breve{\chi}:\ \breve{\mathrm{GL}}_n(\rk)\rightarrow \{\pm 1\}
\]
the quadratic character of $\breve{\mathrm{GL}}_n(\rk)$ projecting to the second factor. 
A tempered generalized function $f$ on $\mathfrak{gl}_n(\rk) \times \rk^{n\times 1} \times \rk^{1\times n}$ is said to be $\breve{\chi}$-equivariant if for every $g\in \breve{\mathrm{GL}}_n(\rk)$, 
\[
f(g.x)=\breve{\chi}(g) f(x), \quad x\in \mathfrak{gl}_n(\rk) \times \rk^{n\times 1} \times \rk^{1\times n}.
\]
The space of such generalized functions is denoted by
\[
\mathrm{C}^{-\xi}_{\breve{\chi}}(\mathfrak{gl}_n(\rk)\times \rk^{n\times 1}\times \rk^{1\times n}).
\]
Similar notation will be used later on without further explanation. Then Proposition \ref{prop:algebra} is equivalent to 
\begin{equation}\label{eq:refor1}
\mathrm{C}^{-\xi}_{\breve{\chi}}(\mathfrak{gl}_n(\rk)\times \rk^{n\times 1}\times \rk^{1\times n}) = 0.
\end{equation}

\vspace{3mm}

\noindent \textbf{Proof of Proposition \ref{prop:group}.} 
Let us form a quadratic extension 
\[
\breve{\mathbb{H}}=(\mathrm{GL}_n(\rk)\times \mathrm{GL}_n(\rk))\rtimes \{ \pm1\}
\]
of the doubling group $\mathrm{GL}_n(\rk)\times \mathrm{GL}_n(\rk)$, where the semi-direct product is defined by the action 
\[
-1.(h_1,h_2) := (h_2^{-t}, h_1^{-t}), \quad h_1,h_2\in \mathrm{GL}_n(\rk).
\]
Let $\breve{\mathbb{H}}$ act on $\mathrm{GL}_n(\rk')$ by 
\[
(h_1,h_2,\delta).g := \begin{cases}
h_1gh_2^{-1}, &\text{if } \delta=1;\\
h_1g^th_2^{-1}, &\text{if } \delta=-1,
\end{cases}
\]
and on $\rk^{n\times 1}\times \rk^{1\times n}$ by 
\[
(h_1,h_2,\delta).(u,v) := \begin{cases}
(h_1u, v h_2^{-1}), &\text{if } \delta=1;\\
(h_1 v^t, u^t h_2^{-1}), &\text{if } \delta=-1,
\end{cases}
\]
where $(h_1,h_2,\delta)\in \breve{\mathbb{H}}$, $g\in \mathrm{GL}_n(\rk')$, and $(u,v)\in \rk^{n\times 1}\times \rk^{1\times n}$. 
Then Proposition \ref{prop:group} can be reformulated as the equality 
\begin{equation}\label{eq:gprefor}
\mathrm{C}^{-\xi}_{\breve{\eta}}(\mathrm{GL}_n(\rk')\times \rk^{n\times 1}\times \rk^{1\times n}) = 0.
\end{equation}
Here $\breve{\eta}$ denotes the quadratic character of $\breve{\mathbb{H}}$ defined by  
\[
\breve{\eta}(g_1,g_2,\delta) = \delta, \quad (g_1,g_2,\delta)\in \breve{\mathbb{H}},
\]
and $\breve{\mathbb{H}}$ acts diagonally on $\mathrm{GL}_n(\rk')\times \rk^{n\times 1}\times \rk^{1\times n}$. 

An element $x\in \mathrm{GL}_n(\rk')$ is said to be $\breve{\mathbb{H}}$-semisimple if its orbit $\breve{\mathbb{H}}.x$ in $\mathrm{GL}_n(\rk')$ is closed. 
By (a version of) the generalized Harish-Chandra descent (see \cite[Theorem 2.2.15]{AG}), 
the equality \eqref{eq:gprefor} is implied by the following assertion:
\begin{equation}\label{eq:descentform}
\mathrm{C}^{-\xi}_{\breve{\eta}_x} (\mathrm{N}_{\breve{\mathbb{H}}.x,x}^{\mathrm{GL}_n(\rk')}\times \rk^{n\times 1}\times \rk^{1\times n}) = 0
\end{equation}
for all $\breve{\mathbb{H}}$-semisimple elements $x$ in $\mathrm{GL}_n(\rk')$. 
Here, $\breve{\eta}_x$ denotes the restriction of $\breve{\eta}$ to the stabilizer $\breve{\mathbb{H}}_x\subset \breve{\mathbb{H}}$ of $x$, and 
\[
\mathrm{N}_{\breve{\mathbb{H}}.x,x}^{\mathrm{GL}_n(\rk')} := \mathrm{T}_x(\mathrm{GL}_n(\rk')) / \mathrm{T}_x(\breve{\mathbb{H}}.x)
\]
denotes the normal space, where $\mathrm{T}_x$ indicates the tangent space at $x$. 

Write $\gamma$ for the involution of $\mathrm{GL}_n(\rk')$ given by the nontrivial element of the Galois group of $\rk'/\rk$. 
By \cite[Lemma 4.2]{CS} (see also \cite[Lemma 7.1.2]{AGS}), it suffices to prove the equality \eqref{eq:descentform} in the case when $x$ is \emph{normal} in the sense that 
\[
x\gamma(x) = \gamma(x)x.
\]
Thus, assume that $x$ is a $\breve{\mathbb{H}}$-semisimple normal element in $\mathrm{GL}_n(\rk')$. Set 
\[
s := x(\gamma(x))^{-1} \in \mathrm{GL}_n(\rk).
\]
Then the stabilizer $\breve{\mathrm{GL}}_n(\rk)_s$ of $s$ in $\breve{\mathrm{GL}}_n(\rk)$ is isomorphic to $\breve{\mathbb{H}}_x$ via the map  
\begin{equation}\label{eq:gpiso}
(h,\delta) \mapsto \begin{cases} 
(xhx^{-1}, h, 1), &\text{if } \delta=1;\\
(xhx^{-t}, h, -1), &\text{if } \delta=-1.
\end{cases}
\end{equation}

Identify the tangent space $\mathrm{T}_x(\mathrm{GL}_n(\rk'))$ with the Lie algebra $\mathfrak{gl}_n(\rk')$ through left translation, and let $\breve{\mathrm{GL}}_n(\rk)$ act on $\mathfrak{gl}_n(\rk')$ as in \eqref{eq:Hactgl}. Then, via the isomorphism \eqref{eq:gpiso}, the isotropy representation $\mathrm{T}_x(\mathrm{GL}_n(\rk'))$ of $\breve{\mathbb{H}}_x$ is identified with the representation $\mathfrak{gl}_n(\rk')$ of $\breve{\mathrm{GL}}_n(\rk)_s\subset \breve{\mathrm{GL}}_n(\rk)$. 
Furthermore, we have (cf.~\cite[Lemma 4.3]{CS} or \cite[Proposition 7.2.1]{AGS})
\[
\mathrm{N}_{\breve{\mathbb{H}}.x,x}^{\mathrm{GL}_n(\rk')}= 
\frac{\mathfrak{gl}_n(\rk')}{\mathfrak{gl}_n(\rk)+\mathrm{Ad}_{x^{-1}}\mathfrak{gl}_n(\rk)} \cong \mathfrak{gl}_n(\rk)_s := \{X\in \mathfrak{gl}_n(\rk)\mid sXs^{-1}=X\}
\]
as representations of $\breve{\mathrm{GL}}_n(\rk)_s$. Thus the equality \eqref{eq:descentform} is equivalent to 
\begin{equation}\label{eq:finalref}
\mathrm{C}_{\breve{\chi}_s}^{-\xi}(\mathfrak{gl}_n(\rk)_s\times \rk^{n\times 1}\times \rk^{1\times n}) = 0.
\end{equation}
Here $\breve{\chi}_s$ denotes the restriction of $\breve{\chi}$ to the stabilizer $\breve{\mathrm{GL}}_n(\rk)_s$. 

Since $s$ is semisimple (as a matrix) in $\mathrm{GL}_n(\rk)$ (see \cite[Proposition 7.2.1]{AGS}), there exist  positive integers $n_1,\dots,n_r$ and field extensions $\rk_1,\dots,\rk_r$ of $\rk$ such that 
\begin{eqnarray}\label{isogl}
&&\breve{\GL}_n(\rk)_s\\
\nonumber &\cong& \breve{\GL}_{n_1}(\rk_1)\times_{\{\pm 1\}} \breve{\GL}_{n_2}(\rk_2)\times_{\{\pm 1\}} \cdots \times_{\{\pm 1\}}\breve{\GL}_{n_r}(\rk_r)\quad (\textrm{the fiber product}).
\end{eqnarray}
Furthermore, we have a linear isomorphism
\[
\mathfrak{gl}_n(\rk)_s\times \rk^{n\times 1}\times \rk^{1\times n} \cong 
\prod_{i=1}^r \mathfrak{gl}_{n_i}(\rk_i)\times \rk_i^{n_i\times 1}\times \rk_i^{1\times n_i}
\]
that is equivariant with respect to the isomorphism \eqref{isogl}. 
This together with the equality \eqref{eq:refor1} implies that
the equality \eqref{eq:finalref} holds. 
Thus we complete the proof of Proposition \ref{prop:group}. 

\vspace{3mm}

\noindent \textbf{Proof of Theorem \ref{thm:main3}.} 
We apply Theorem \ref{thm:main2} in the following setting:
\begin{itemize}
    \item $G := \mathrm{GL}_n(\rk')$, $H_1 := H_2 := \mathrm{GL}_n(\rk)$;
    \item $\omega_1 := \mathrm{D}^\varsigma(\rk^{n\times 1})$ and $\omega_2 := \mathrm{D}^{\varsigma}(\rk^{1\times n})$;
    \item $\sigma: G \rightarrow G$ is the transpose map;
    \item $\sigma_1: \omega_1 \rightarrow \omega_2$ is the push-forward of the transpose map $\rk^{n\times 1} \rightarrow \rk^{1\times n}$, and $\sigma_2$ is its inverse; 
    \item the Chevalley automorphism $\sigma': G \rightarrow G$ is the inverse transpose map;
    \item $\varphi := \sigma_1: \omega_1 \rightarrow \omega_2$.
\end{itemize}
Here $\mathrm{GL}_n(\rk)$ acts on $\mathrm{D}^\varsigma(\rk^{n\times 1})$ and $\mathrm{D}^{\varsigma}(\rk^{1\times n})$ via left and right multiplications, respectively. 

In view of Proposition \ref{prop:group}, 
Theorem \ref{thm:main2} (applied in the above setting) implies that 
\[
(\mathrm{GL}_n(\rk'), \mathrm{D}^\varsigma(\rk^{n\times 1}))
\]
is a Gelfand pair. 
The theorem then follows from the observation that the Fourier transform from $\mathrm{D}^\varsigma(\rk^{1\times n})$ to $\mathrm{C}^\varsigma(\rk^{n\times 1})$ is $\mathrm{GL}_n(\rk)$-equivariant. 

\begin{remt}
Note that every character of $\mathrm{GL}_n(\rk)$ extends to a character of $\mathrm{GL}_n(\rk')$. It then follows from Theorem \ref{thm:main3} that for every character $\chi$ of $\mathrm{GL}_n(\rk)$,
\[
(\mathrm{GL}_n(\rk'), \mathrm{C}^\varsigma(\rk^{1\times n})\otimes \chi)
\]
is a Gelfand pair. 
\end{remt}
 \section*{Acknowledgments}
 F. Chen was supported in part by the National Natural Science Foundation of China (Grant Nos. 12131018 and 12471029).
 B. Sun was supported in part by National Key R \& D Program of China No. 2022YFA1005300  and New Cornerstone Science foundation.

\end{document}